\documentclass[a4paper,11pt]{amsart}

\usepackage[all]{xy}
\usepackage{amsfonts}
\usepackage{amsmath}
\usepackage{amssymb}
\usepackage{amsthm}
\usepackage{enumerate}

\title{More Reduced Obstruction Theories}
\author{Timo Sch\"urg}
\address{Mathematisches Institut \\Universit\"at Bonn \\Endenicher Allee 60\\53115 Bonn\\Germany}
\email{timo\_schuerg@operamail.com}

\DeclareMathOperator{\tr}{tr}
\DeclareMathOperator{\at}{at}
\DeclareMathOperator{\Perf}{Perf}
\DeclareMathOperator{\St}{St}
\DeclareMathOperator{\Pic}{Pic}
\DeclareMathOperator{\dSt}{dSt}
\DeclareMathOperator{\cofib}{cofib}
\DeclareMathOperator{\QCoh}{QCoh}
\DeclareMathOperator{\fib}{fib}

\DeclareMathOperator{\Spec}{Spec}

\DeclareMathOperator{\der}{der}

\DeclareMathOperator{\Ext}{Ext}
\DeclareMathOperator{\Hom}{Hom}

\DeclareMathOperator{\Mor}{Mor}

\newcommand{\IC}{\mathbb{C}}
\newcommand{\IG}{\mathbb{G}}

\newcommand{\IR}{\mathbb{R}}

\newcommand{\IZ}{\mathbb{Z}}

\newcommand{\sO}{\mathcal{O}}

\newcommand{\RMor}{\IR \negthinspace \Mor}
\newcommand{\RPic}{\IR \negthinspace \Pic}
\newcommand{\RPerf}{\IR \negthinspace \Perf}

\newtheorem{thm}{Theorem}[section]
\newtheorem{lem}[thm]{Lemma}
\newtheorem{prop}[thm]{Proposition}
\newtheorem{cor}[thm]{Corollary}

\theoremstyle{definition}
\newtheorem{defn}[thm]{Definition}
\newtheorem{exmp}[thm]{Example}

\theoremstyle{remark}
\newtheorem{rem}[thm]{Remark}

\def\signed #1{{\leavevmode\unskip\nobreak\hfil\penalty50\hskip2em
    \hbox{}\nobreak\hfil #1%
  \parfillskip=0pt \finalhyphendemerits=0 \endgraf}}

\newsavebox\mybox

\begin{document}
\begin{abstract}
  We first develop a general formalism for globally removing factors 
  from an obstruction theory. We then apply this formalism to give a 
  construction of a reduced obstruction theory on the moduli space of 
  morphisms from a curve to a surface $f \colon C \to S$ in class 
  $\beta$ such that $H^1(C,f^* T_S) \overset{_{-}\cup \beta}{\to} 
  H^2(S,\sO _S)$ is surjective.  This condition appears in recent work 
  of Kool and Thomas.
\end{abstract}
\maketitle
%%%%%%%%%%
\section{Introduction}
%%%%%%%%%%

Due to the deformation invariance of Gromov-Witten invariants, smooth 
complex projective surfaces having deformations in a direction where the 
topological class $\beta \in H^{1,1}(S, \IC) \cap H^2(S,\IZ)$ does not 
stay of type $(1,1)$ have no Gromov-Witten invariants.

This can be fixed by introducing a \emph{reduced perfect obstruction 
  theory} for the moduli of stable maps. This obstruction theory is 
obtained by removing a factor from the usual obstruction theory, and is 
only invariant under deformations for which $\beta$ stays of the given 
type. This technique has been used extensively in the case of 
$K3$--surfaces by Maulik--Pandharipande \cite{maulik} and by 
Maulik--Pandharipande--Thomas \cite{mpt}. It has proven difficult though 
to show that the complex obtained by removing a factor from the standard 
obstruction theory indeed satisfies all requirements for being an 
obstruction theory. This is closely related to the problem of showing 
that realized obstructions lie in the kernel of a semi-regularity map.  
To treat other kinds of surfaces as well, an alternative approach has 
been introduced recently by Kool and Thomas \cite{thomas}.  There a 
complex having all formal properties expected from the reduced 
obstruction theory is defined using an \emph{algebraic twistor family}.  
The aim of this note is to construct a reduced obstruction theory in 
these cases directly without resorting to an algebraic twistor family. 

The basic idea how to remove a factor from an obstruction theory 
presented here is by mapping the moduli problem in question to a further 
smooth moduli problem, which nonetheless has a non-trivial obstruction 
theory. In the case of moduli of sheaves on a surface this was already 
used by Mukai \cite{mukai} and Artamkin \cite{artamkin}. The difficult 
part is to ensure that the obstruction theories of the two moduli 
problems are compatible.  In the context of formal moduli problems, it 
was observed by Manetti \cite{manetti} that this is automatic as soon as 
the morphism of formal moduli problems is induced by a morphism of 
differential graded Lie algebras. This technique allows us to show that 
obstructions to a formal moduli problem lie in the kernel of the induced 
map of the obstruction spaces.

Knowing the compatibilities of the obstruction theories of the formal 
moduli problems corresponding to closed points of the global moduli 
space is not enough though to globally remove a factor. As an example 
for the amount of calculations necessary, see the example of 
Donaldson--Thomas invariants \cite[Section 3]{casson}, or for the case 
of Pandharipande--Thomas invariants \cite{huybrechts}. In both cases the 
obstructions have to be considered with respect to a fixed determinant.  
It requires compatibilities of the obstructions not only over 
square-zero extensions of Artinian rings, but over square-zero 
extensions of aribtrary bases. 

In this note we study the case where a strong compatibility of the 
obstruction theories is available globally on the moduli space. As 
compatibility datum between the obstruction theories we require more 
than just commutativity in the derived category. Instead, we assume that 
the diagram of obstruction theories commutes up to homotopy in some 
higher categorical model. We show that in this case a factor of the 
obstruction theory can be removed globally on the entire moduli space in 
question. Using commutativity up to homotopy instead of commutativity in 
the derived category and cofiber sequences instead of exact triangles 
makes the necessary calculations a breeze.

This raises the question where such strong compatibilities can be found. 
Natural examples where such compatibilities up to homotopy are available 
come from derived algebraic geometry. In a previous work \cite{stv}, a 
morphism from the derived moduli space of stable maps to the derived 
Picard stack was introduced.  Assuming the compatibility of the 
obstruction theories obtained from this morphism, we show that a factor 
of the obstruction theory can be removed globally.

\subsection*{Conventions}
We have tried to adhere to the following conventions. We work throughout 
over an arbitrary base ring $k$, which in Section 4 becomes the field of 
complex numbers. We will denote the cotangent complex of a scheme, or 
more generally an Artin stack, over $k$ by $L_X$ instead of $L_{X/k}$.  
Contrary to what is common in algebraic geometry, we have used 
homological grading.  Finally we will denote by $\QCoh (X)$ the 
$\infty$-category of quasi-coherent complexes constructed by Lurie in 
\cite{DAGVIII}. The reason for employing this category instead of the 
derived category is that at certain points it is important to know why 
things are homotopic, and not only that they are homotopic. It also 
allows to carry out proofs as if one was only dealing with modules, and 
not with complexes. Recall that a \emph{cofiber sequence} in $\QCoh (X)$ 
consists of a sequence of morphisms $E \overset{f}{\to} F 
\overset{g}{\to} G$, a 2-simplex identifying the composition $fg$ with a 
morphism $E \overset{h}{\to} G$, and a nullhomotopy of $h$. The 
$\infty$-category $\QCoh (X)$ is equipped with a standard $t$-structure.  
We will use that the notion of Tor-amplitude behaves well with respect 
to this $t$-structure, i.e., if an object $E \in \QCoh (X)$ is of 
Tor-amplitude $\leq n$, then $E[m]$ is of Tor-amplitude $\leq n+m$. All 
details can be found in Lurie's volumes \cite{HigherAlgebra,DAGVIII}.

I would like to thank Richard Thomas for suggesting the subject of this 
short note and emphasizing the importance of Lemma \ref{lem:compzero}, 
Gabriele Vezzosi for helpful discussions, as well as Daniel Huybrechts 
for helpful comments.  The dependence of this note on material from 
\cite{stv} is obvious.

This work was supported by the SFB/TR 45 `Periods, Moduli Spaces and 
Arithmetic of Algebraic Varieties' of the DFG (German Research 
Foundation).
%%%%%%%%%%
\section{Removing Factors}
%%%%%%%%%%

We first introduce the geometric objects we wish to study. These are in 
general Artin stacks with a fixed 1-perfect obstruction theory. The 
terminology \emph{virtually smooth} for a pair of an Artin stack 
together with a fixed perfect obstruction theory was introduced by 
Fantechi--G\"ottsche in \cite{fg}.

\begin{defn}
  A pair $(X, \phi \colon E \to L_X)$ is a \emph{virtually smooth Artin  
    stack} if $X$ is an Artin stack locally of finite type over $k$, $E$ 
  is a perfect complex of Tor-amplitude $\leq 1$, and $\cofib(\phi) \in 
  \QCoh (X) _{\geq 2}$. The morphism $\phi$ will be referred to as the 
  \emph{obstruction theory}.
\end{defn}

If $X$ is a Deligne--Mumford stack, the morphism $\phi \colon E \to L_X$ 
in the above definition is a 1-perfect obstruction theory in the sense 
of Behrend--Fantechi \cite{behrend97}. We next define morphisms between 
such objects.

\begin{defn}
  A \emph{morphism of virtually smooth Artin stacks} is a pair 
  \[
    (f,\alpha) \colon (X, \phi \colon E \to L_X) \longrightarrow (Y, 
    \chi \colon F \to L_Y)
  \]
  where $f \colon X \to Y$ is a morphism of Artin stacks over $k$, and 
  $\alpha \colon f^*F \to E$ is a morphism of perfect complexes on $X$ 
  such that
  \[
    \xymatrix{
      f^* F \ar[r]^{\chi} \ar[d]_{\alpha} & f^* L_Y \ar[d] \\
      E \ar[r]^{\phi} \ar[r] & L_X
    }
  \]
  commutes in $\QCoh (X)$.
\end{defn}

\begin{rem}
  Recall that commuting in $\QCoh (X)$ means that we have fixed a 
  homotopy making the diagram commutative. This added information is 
  absolutely essential for all further computations.
\end{rem}

We will also need the notion of virtually smooth morphism.

\begin{defn}
  Let $(f,\alpha) \colon (X, \phi \colon E \to L_X) \longrightarrow (Y, 
  \chi \colon F \to L_Y)$ be a morphism of virtually smooth Artin  
  stacks.  Then $(f,\alpha)$ is a \emph{virtually smooth morphism} if 
  $\cofib (\alpha)$ is of Tor-amplitude $\leq 1$.
\end{defn}

\begin{rem}
  Note that a priori $\cofib (\alpha)$ is only of Tor-amplitude $\leq 
  2$.
\end{rem}

Derived algebraic geometry provides natural examples of virtually smooth 
Artin stacks and morphisms between these.

\begin{exmp}
  \label{exmp:der}
  Recall that a derived Artin  stack $X^d$ over $k$ is 
  \emph{quasi-smooth} if its cotangent complex $L_{X^d}$ is of 
  Tor-amplitude $\leq 1$ and its underlying Artin  stack $X:=t_0(X^d)$ 
  is locally of finite type over $k$. By the canonical inclusion $j_X 
  \colon X \hookrightarrow X^d$ we can obtain the structure of a 
  virtually smooth Artin  stack on $X$. The perfect obstruction theory 
  is given by the canonical morphism $\phi \colon j_X ^* L_{X^d} \to 
  L_X$.  Using the functoriality properties of cotangent complexes, 
  every morphism of quasi-smooth derived Artin  stacks gives rise to a 
  morphism of virtually smooth Artin  stacks.
\end{exmp}

To remove factors from obstruction theories we will make use of 
virtually smooth Artin stacks with some peculiar properties. We will be 
using virtually smooth Artin stacks $(Y, \chi \colon F \to L_Y)$ where 
the underlying Artin stack $Y$ itself is already smooth. This is 
contrary to the philosophy that spaces become smooth after deriving 
them, or that 1-perfect obstruction theories are only interesting on 
very singular spaces. On the contrary, we will have to find non-smooth 
derived versions of spaces that are smooth, or 1-perfect obstruction 
theories on smooth spaces.

\begin{defn}
  Let $(f,\alpha) \colon (X, \phi \colon E \to L_X) \longrightarrow (Y, 
  \chi \colon F \to L_Y)$ be a morphism of virtually smooth schemes. We 
  say that $(f, \alpha)$ is a \emph{reduction morphism} if $Y$ is 
  smooth.
\end{defn}

\begin{rem}
  For applications to virtual classes, $X$ will assumed to be a 
  Deligne--Mumford stack.
\end{rem}

Given a reduction map, we would like to define a new structure of 
virtually smooth Artin stack on $X$ such that the virtual dimension of 
$X$ increases. The factor we would like to remove from the obstruction 
theory $E$ is the pull-back to $X$ of the fiber of $\chi \colon F \to 
L_Y$.  In the following we will show that this is possible if the 
reduction morphism is virtually smooth. The key to removing a factor is 
the following lemma, which is true in much greater generality than we 
actually need. Note that we do not assume $(f,\alpha)$ either to be a 
reduction morphism or virtually smooth.

\begin{lem}
  \label{lem:compzero}
  Let $(f,\alpha) \colon (X, \phi \colon E \to L_X) \longrightarrow (Y, 
  \chi \colon F \to L_Y)$ be a morphism of virtually smooth schemes. Let 
  $K = \fib (\chi)$, and define $\beta$ to be the composition
  \[
    f^* K \overset{\gamma}{\longrightarrow} f^*F 
    \overset{\alpha}{\longrightarrow} E.
  \]
  Then the composition
  \[
    f^*K \overset{\beta}{\longrightarrow} E 
    \overset{\phi}\longrightarrow L_X
  \]
  is zero.
\end{lem}
\begin{proof}
  By definition, we have a cofiber sequence $K \to F \to L_Y$, and this 
  remains a cofiber sequence after pulling to $X$. We thus have the 
  following commutative diagram on $X$
  \[
    \xymatrix{
      f^*K \ar[r] \ar[d]_{\gamma} & 0 \ar[d] \\
      f^*F \ar[r]^{\chi} \ar[d]_{\alpha} & f^*L_{Y} \ar[d] \\
      E \ar[r]^{\phi}  & L_{X}
    }
  \]
  which gives a homotopy from $\phi \circ \beta$ to zero.
\end{proof}

We can now define our candidate for a reduced obstruction theory. We let 
$E' := \cofib(\beta)$. By Lemma \ref{lem:compzero}, we have a 
well-defined morphism $\phi' \colon E' \to L_X$. Note that if we only 
knew the composition to be zero in the derived category this would not 
be sufficient to obtain a well-defined morphism.

\begin{thm}
  \label{thm:main}
  Let $(f,\alpha) \colon (X, \phi \colon E \to L_X) \longrightarrow (Y, 
  \chi \colon F \to L_Y)$ be a virtually smooth reduction map. Then
  \[
    (X, \phi' \colon E' \to L_{X})
  \]
  is a virtually smooth Deligne--Mumford stack.
\end{thm}
\begin{proof}
  We first show that $E'$ is perfect. Let $K$ as above denote $\fib 
  (\chi)$, so that we have a cofiber sequence
  \[
    K \longrightarrow F \overset{\chi}{\longrightarrow} L_Y.
  \]
  Now $F$ is perfect by assumption, and $L_Y$ is perfect since $Y$ is 
  smooth and locally of finite presentation. Since the property of being 
  perfect is stable under cofiber sequences, $K$ is perfect, and thus 
  $f^*K$ is perfect. This shows that $E'$ is the cofiber of a morphism 
  between perfect objects, and thus is perfect.

  We now want to show that $E'$ is of Tor-amplitude $\leq 1$. Since $Y$ 
  is smooth and $L_Y$ thus is of Tor-amplitude $\leq 0$, the above 
  cofiber sequence shows that $K$ is of Tor-amplitude $\leq 1$. It 
  follows that $f^*K$ is also of Tor-amplitude $\leq 1$. Let $\gamma$ 
  denote the morphism $f^*K \to f^*F$. By definition, the diagram
  \[
    \xymatrix{
      f^*K \ar[r]^{\gamma} \ar[dr]_{\beta}& f^* F \ar[d]^{\alpha} \\
       & E
     }
   \]
   commutes. This gives us a cofiber sequence $\cofib (\gamma) \to 
   \cofib (\beta) \to \cofib (\alpha)$. Since $E'=\cofib (\beta)$ and 
   $f^* L_Y = \cofib (\gamma)$, we have a cofiber sequence
   \[
     f^* L_Y \longrightarrow E' \longrightarrow \cofib (\alpha).
   \]
  Since we assumed $(f,\alpha)$ a virtually smooth morphism, $\cofib 
  (\alpha)$ is of Tor-dimension $\leq 1$. Again using that $Y$ is 
  smooth, it follows that $E'$ is of Tor-dimension $\leq 1$.
  
  It remains to show that $\cofib(\phi') \in \QCoh(X)_{\geq 2}$, or 
  equivalently that $\fib(\phi') \in \QCoh(X)_{\geq 1}$. Let $K' = \fib 
  (\phi)$. Since the composition $\phi \circ \beta$ factors over zero, 
  we obtain a morphism $\delta \colon  f^*K \to K'$, and we can identify 
  $\fib(\phi')$ with $\cofib (\delta)$. Since $ f^*K$ and $K'$ are both 
  in $\QCoh (X) _{\geq 1}$, the claim follows.
\end{proof}

\begin{rem}
  Since $\left( X, \phi' \colon E' \to L_X \right)$ is a virtually 
  smooth Artin stack this automatically poses the question if this 
  obstruction theory is induced by a natural structure of a derived 
  Artin stack on $X$.  Adding plenty of assumptions such a statement 
  indeed holds.  First of all, we have to assume that the perfect 
  obstrution theories $\left( E \to L_X \right)$ and $\left( F \to L_Y 
  \right)$ are induced by derived stacks $X^d$ and $Y^d$, and the 
  compatibility datum $\alpha$ is induced by a morphism $f^d \colon X^d 
  \to Y^d$.  Furthermore, we have to assume that the derived structure 
  on $Y^d$ splits as
  \[
    Y^d = Y \times Y^{\der}.
  \]
  The underling stack of $Y^{\der}$ is a point. Let $p\colon Y^d \to 
  Y^{\der}$ be the projection. The homotopy fiber product of the diagram
  \[
    \xymatrix{
      & X^d \ar[d]^{p \circ f^d} \\
      \Spec k \ar[r] & Y^{\der}
    }
  \]
  then yields the desired derived Artin stack. Such a splitting exists 
  for the derived Picard stack of a $K3$--surface. It is reasonable to 
  expect such a splitting to exist whenever $Y^d$ is a group stack with 
  smooth truncation.
\end{rem}
%%%%%%%%%%
\section{Application to Deformation Theory}
%%%%%%%%%%

In the following assume that $(X, \phi \colon E \to L_X)$ is a virtually 
smooth Deligne--Mumford stack. Given a reduction map $(f,\alpha) \colon 
(X, \phi \colon E \to L_X) \longrightarrow (Y, \chi \colon F \to L_Y)$, 
we can define a \emph{generalized semi-regularity map}. Given a morphism 
$p \colon T \to X$ where $T=\Spec (A)$ is an affine scheme, and a 
square-zero extension $T \hookrightarrow T'$ classified by a morphism 
$\eta \colon L_T \to M[1]$ for some $A$-module $M$, $p \colon T \to X$ 
extends to a morphism $T' \to X$ if and only if the element in 
$\Ext^1(p^*E,M)$ defined by the homotopy class of the composition
\[
  p^*E \overset{p^*\phi}{\longrightarrow} p^*L_X \longrightarrow L_T 
  \overset{\eta}{\longrightarrow} M[1]
\]
  vanishes. We define the generalized semi-regularity map to be the map
\[
  \Ext^1(p^*E,M) \longrightarrow \Ext ^1 (p^*f^*F,M)
\]
obtained by composition with $\alpha$. We will now show that realized 
obstructions lie in the kernel of the generalized semi-regularity map.  
We first give a definition of realized obstructions following 
Behrend--Fantechi \cite{behrend97}.

\begin{defn}
  \label{defn:RealObs}
  Let $(X, \phi \colon E \to L_X)$ be a virtually smooth Deligne--Mumford 
  stack,  and let $p \colon T \to X$ be a morphism with $T=\Spec (A)$.  
  Let $M$ be a $A$-module. A non-zero morphism $\alpha \colon p^*E \to 
  M[1]$ \emph{realizes an obstruction} if there exists a square-zero 
  extension $T \hookrightarrow T'$ classified by $\eta \colon L_T \to 
  M[1]$ such that
  \[
    \xymatrix{
      p^*E \ar[r] \ar[rrd]_{\alpha} & p^*L_X \ar[r] & L_T 
      \ar[d]^{\eta}\\
      & & M[1]
    }
  \]
  commutes.
\end{defn}

We can now show that obstructions that are actually realized always lie 
in the kernel of the generalized semi-regularity map.

\begin{prop}
  \label{prop:vanish}
  Let $(f,\alpha) \colon (X, \phi \colon E \to L_X) \longrightarrow (Y, 
  \chi \colon F \to L_Y)$ be a reduction morphism. Assume that $X$ is a 
  Deligne--Mumford stack, and let $p \colon T \to X$ be a morphism where 
  $T$ is an affine scheme. Then realized obstructions lie in the kernel 
  of the generalized semi-regularity map.
\end{prop}
\begin{proof}
  Since $Y$ is smooth, this allow us to conclude that 
  $\Ext^1(p^*f^*L_Y,M)$ and $\Ext^2(p^*f^*L_Y,M)$ are zero. Using the 
  pull-back of the cofiber sequence
  \[
    K \longrightarrow F \overset{\chi}\longrightarrow L_Y
  \]
  to $T$ we thus have an isomorphism
  \[
    \Ext^1(p^*f^*F,M) \simeq \Ext^1(p^*f^*K,M).
  \]
  Applying Lemma \ref{lem:compzero} the claim follows.
\end{proof}

\begin{rem}
  A reduction morphism is virtually smooth if its generalized 
  semi-regularity morphism is surjective.
\end{rem}

%%%%%%%%%%
\section{Application to moduli of maps}
%%%%%%%%%%%%%%%%%%%%%%

We now apply the formalism developed above in an example, working over 
$k=\IC$. The example we will be concerned with is the moduli space of 
maps from a fixed curve $C$ to a smooth projective complex surface $S$ 
satisfying the condition $c_1(\IR f_* \sO_C)=\beta$, where $\beta \in 
H^1(S,\Omega^1_S)$. We will denote this space by $\Mor_{\beta} (C,S)$.  
It is well-known that this space is actually a virtually smooth scheme, 
see Behrend--Fantechi \cite{behrend97}.  We will denote this space 
equipped with its standard obstruction theory by
\[
  \left( \Mor_{\beta} (C,S), \phi \colon E \to L_{\Mor_{\beta} 
      (C,S)}\right).
\]

To apply the results of \cite{stv}, it is important to note that the 
same structure of virtually smooth scheme can also be constructed using 
Example \ref{exmp:der}. To see this, denote by $i \colon \St_k \to \dSt 
_k$ the inclusion functor from stacks over $k$ to derived stacks over 
$k$. We can then define the derived moduli space of maps to be the 
derived scheme parametrizing morphisms in this larger category. We will 
denote this derived scheme by $ \RMor _{\beta}(C,S)$.

In order to remove a factor from the obstruction theory using the above 
formalism we have to find some virtually smooth Artin stack as 
comparison space. The natural candidate in this example is the Picard 
stack $\Pic (S):= \underline{\Hom}_{\St_k}(S,B \IG _m)$. As above, there 
again is a derived version of this stack, given by $\RPic (S) := 
\underline{\Hom}_{\dSt_k}(S,B \IG _m)$. Denote the canonical inclusion 
by $j \colon \Pic (S) \to \RPic (S)$. The virtually smooth space we will 
use as target for our potential reduction morphism is
\[
  \left( \Pic (S), \chi \colon  j^*L_{\RPic (S)} \to L_{\Pic (S)} 
  \right).
\]
Since the underlying Artin stack $\Pic (S)$ is smooth, this is an 
excellent candidate for a reduction map.

Finally, we have to give a map of virtually smooth schemes. In 
\cite{stv}, a map
\[
  \RMor_{\beta}(C,S) \overset{A_S}{\longrightarrow} \RPerf (S) 
  \overset{\det}{\longrightarrow} \RPic (S)
\]
is given. Using Example \ref{exmp:der}, we obtain a map of virtually 
smooth schemes
\begin{align*}
  (f,\alpha) \colon \left( \Mor_{\beta} (C,S), \phi \colon E \to 
    L_{\Mor_{\beta} (C,S)}\right)& \longrightarrow \\
   & \left( \Pic (S), \chi \colon  j^*L_{\RPic (S)} \to L_{\Pic (S)} 
   \right) .
\end{align*}

\begin{rem}
  The generalized semi-regularity map associated to $(f,\alpha)$ is just 
  the semi-regularity map for morphisms of Buchweitz--Flenner 
  \cite[Remark   7.24]{buchweitz}.
\end{rem} 

We can now define a new structure of virtually smooth scheme on 
$\Mor_{\beta}(C,S)$. Let $K:= \fib (\chi)$. Note that $K$ is non-trivial 
if and only if $H^2(S,\sO _S)$ is non-trivial. As above, we then have a 
morphism $\gamma \colon f^* K \to E$, and can define $E' := \cofib 
(\gamma)$ as candidate for a reduced obstruction theory.

\begin{cor}
  \label{cor:red}
  Assume that
  \begin{align*}
    (f,\alpha) \colon \left( \Mor_{\beta} (C,S), \phi \colon E \to 
      L_{\Mor_{\beta} (C,S)}\right)& \longrightarrow \\
     & \left( \Pic (S), \chi \colon  j^*L_{\RPic (S)} \to L_{\Pic (S)} 
     \right) .
  \end{align*}
  is virtually smooth. Then
  \[
    \left( \Mor_{\beta}(C,S), \phi' \colon E' \to L_{\Mor_{\beta}(C,S)} 
    \right)
  \]
  is a virtually smooth scheme.
\end{cor}

\begin{rem}
  Note that the two structures of virtually smooth scheme on 
  $\Mor_{\beta}(C,S)$ only differ in case $H^2(S,\sO _S)$ is non-zero.
\end{rem}

\begin{exmp}
  Assume that $S$ is a $K3$--surface. Then the morphism $(f,\alpha)$ is 
  virtually smooth for any class $\beta \neq 0$.
\end{exmp}

We finally want to state a condition ensuring that $(f, \alpha)$ is 
virtually smooth. This condition was identified by Kool and Thomas 
\cite{thomas} and provided the motivation for this work.

\begin{prop}
  Assume that
  \[
    H^{1}(S,T_S) \overset{\cup \beta}{\longrightarrow} H^2(S,\sO _S)
  \]
  is surjective. Then $(f,\alpha)$ is virtually smooth.
\end{prop}
\begin{proof}
  It suffices to prove the statement on $k$-points. We thus have to show 
  that at any point $p \colon \Spec k \to  \Mor_{\beta}(C,S)$ the 
  morphism
  \[
    \pi_1(p^* \alpha) \colon  \pi_1(p^* L_{\IR \Pic (S)}) 
    \longrightarrow \pi_1(p^*E)
  \]
  is injective. Equivalently, we have to show that the dual of 
  $\pi_1(p^* \alpha)$ is surjective.
  
  Let $g \colon C \to S$ be the morphism corresponding to $p$. Recall 
  from Illusie \cite[Chapitre V]{illusie} or explicitly from 
  \cite{huybrechts} that for any perfect complex the first Chern class 
  factors as composition of the Atiyah class and the trace-map. Let 
  $E=\IR g_* \sO_C$. Thus, for any class $\alpha \in H^1(S,T_S)$ the 
  operation of cup-product with $\beta=c_1(E)$ factors as
  \[
    \xymatrix{
      H^1(S,T_S) \ar[dr]^{_{-} \cup \beta} \ar[d]_{_{-} \cup \at_E} & \\
      \Ext^2_S(E,E) \ar[r]_{\tr} & H^2(S,\sO_S).
    }
  \]
  Now in \cite[Appendix]{stv} it is shown that $_{-} \cup \at_E$ factors 
  as
  \[
    \xymatrix{
      & H^1(S,T_S) \ar[d]^{_{-} \cup \at_E} \ar[dl] \\
      H^1(C,g^* T_S) \ar[r]_{T_{A_S,g}} & \Ext_S^2(E,E).
    }
  \]
  Here $T_{A_S,g}$ is the tangent to
  \[
  A_S \colon \RMor_{\beta}(C,S) \longrightarrow \RPerf (S)
  \]
  at the point $p$. Piecing the two diagrams together, we arrive at a 
  commutative diagram
  \[
    \xymatrix{
      &  H^1(S,T_S) \ar[dr]^{_{-} \cup \beta} \ar[d]_{_{-} \cup \at_E} 
      \ar[dl]& \\
      H^1(C,g^* T_S) \ar[r]_{T_{A_S}} & \Ext^2_S(E,E) \ar[r]_{\tr} & 
      H^2(S,\sO_S).
    }
  \]
  Since the bottom row is the dual of $\pi_1(p^*\alpha)$ and by 
  assumption
  \[
    H^{1}(S,T_S) \overset{\cup \beta}{\longrightarrow} H^2(S,\sO _S)
  \]
  is surjective, the claim follows.
\end{proof}

\begin{rem}
  Behrend and Fantechi in \cite{oberwolfach} suggested removing a factor 
  of $H^0(X, \Omega^2 _X)$ from the obstruction theory of the moduli 
  space of stable maps to an irreducible complex symplectic variety of 
  dimension $n$ to perform refined curve counts. The formalism developed 
  here applies as soon as one has an appropriate target for a reduction 
  morphism. Promising candidates are the derived version of the 
  intermediate Jacobian $J^p_X$ with $p=n-1$ constructed recently by 
  Pridham \cite{pridham} and Iacono--Manetti \cite{iacono}. More 
  generally, this should work for any variety for which an analogue of 
  the surjectivity of cup-product with $\beta$ holds.
\end{rem}

%%%%%%%%%%
\bibliographystyle{amsplain}
\bibliography{bibliography}
\end{document}